\documentclass[11pt,a4paper,onesided]{article}

\usepackage{ amssymb, amsmath, amsthm, graphics, xspace, enumerate}
\usepackage{graphicx}
\usepackage[a4paper,colorlinks=true,citecolor=blue,urlcolor=blue,linkcolor=blue,bookmarksopen=true,unicode=true,pdffitwindow=true]{hyperref}
 \hypersetup{pdfauthor={Nikola Sandri\'{c}}}
\hypersetup{pdftitle={Long-time behavior  of stable-like processes}}

\newtheorem{tm}{tm}[section]
\newtheorem{theorem}[tm]{Theorem}
\newtheorem{lemma}[tm]{Lemma}
\newtheorem{corollary}[tm]{Corollary}

\newtheorem{definition}[tm]{Definition}

\newcommand {\R} {\ensuremath{\mathbb{R}}}
\newcommand {\ZZ} {\ensuremath{\mathbb{Z}}}
\newcommand {\N} {\ensuremath{\mathbb{N}}}
\newcommand {\CC} {\ensuremath{\mathbb{C}}}
\newcommand{\process}[1]{\{#1_t\}_{t\geq0}}
\newcommand{\chain}[1]{\{#1_n\}_{n\geq0}}
\numberwithin{equation}{section}
\def\be{\begin{equation}}
\def\ee{\end{equation}}
\begin{document}

 \title{Long-time behavior of stable-like processes}
 \author{Nikola Sandri\'{c}\\ Department of Mathematics\\
         Faculty of Civil Engineering, University of Zagreb, Zagreb,
         Croatia \\
        Email: nsandric@grad.hr }

 \maketitle
\begin{center}
{
\medskip

} \end{center}

\begin{abstract}
In this paper, we consider a long-time behavior of  stable-like
processes. A stable-like process is a  Feller process given by the
symbol $p(x,\xi)=-i\beta(x)\xi+\gamma(x)|\xi|^{\alpha(x)},$ where
$\alpha(x)\in(0,2)$, $\beta(x)\in\R$ and $\gamma(x)\in(0,\infty)$.
More precisely, we give sufficient conditions for recurrence,
transience and ergodicity of stable-like processes in terms of the
stability function $\alpha(x)$,  the drift function $\beta(x)$ and
the scaling function $\gamma(x)$. Further, as a special case of
these results we give a new proof for the recurrence and transience
property of one-dimensional symmetric stable L\'{e}vy processes with
the index of stability $\alpha\neq1.$

\end{abstract}


%
%
%
%


\section{Introduction}

\quad \ \ The recurrence and transience property  of L\'{e}vy
processes, and in particular of stable L\'{e}vy processes, has been
completely studied in \cite[Chapter 7]{sato-book}. In this paper, we
consider  long-time behavior
 of stable-like processes.
Let $\alpha:\R\longrightarrow(0,2)$, $\beta:\R\longrightarrow\R$ and
$\gamma:\R\longrightarrow(0,\infty)$ be arbitrary bounded and
continuously differentiable functions with bounded derivatives, such
that $0<\inf\{\alpha(x):x\in\R\}\leq\sup\{\alpha(x):x\in\R\}<2$ and
$0<\inf\{\gamma(x):x\in\R\}$. Under this assumptions, R. Bass
\cite{bass-stablelike}, R.L.
                                                         Schilling
                                                         and J. Wang
                                                         \cite[Theorem 3.3.]{rene-wang-feller}
                                                         have shown
                                                         that there
                                                         exists a unique
                                                         Feller
                                                         process
                                                         called a
                                                         \emph{stable-like
                                                         process},
                                                         which we
                                                         denote by
                                                         $\process{X^{\alpha}}$,
                                                         given by
                                                         the
                                                         following
                                                         infinitesimal
                                                         generator \begin{align}\label{eq:1.1}\mathcal{A}^{\alpha}f(x)=\beta(x) f'(x)+\int_{\R}\left(f(y+x)-f(x)-f'(x)y1_{\{z:|z|\leq1\}}(y)\right)\frac{c(x)}{|y|^{\alpha(x)+1}}dy,\end{align}
                                                where
                                                $$c(x)=\gamma(x)\frac{\alpha(x)
                 2^{\alpha(x)-1}\Gamma(\frac{\alpha(x)+1}{2})}{\pi^{\frac{1}{2}}\Gamma(1-\frac{\alpha(x)}{2})}.$$
  Clearly, the symbol of this process is given by
$p(x,\xi)=-i\beta(x)\xi+\gamma(x)|\xi|^{\alpha(x)}.$ The aim of this
paper is to find sufficient conditions for  recurrence, transience
and ergodicity of  stable-like processes. The main tool used in
proving these conditions is  the Foster-Lyapunov criteria for
general Markov processes, which were developed in
\cite{meyn-tweedie-III}.

Long-time behavior of stable-like processes has already been
considered in  literature. Clearly, in the case when $\alpha(x)$,
$\beta(x)$ and $\gamma(x)$ are constant functions, the stable-like
process $\process{X^{\alpha}}$ becomes  one-dimensional symmetric
stable L\'{e}vy process with the drift. By the Chung-Fuchs criterion
(see \cite[Corollary 37.17]{sato-book}), its recurrence and
transience property depends only on the index of stability
$\alpha\in(0,2]$ and the drift $\beta\in\R$. More precisely,
one-dimensional symmetric stable L\'{e}vy process with the drift is
recurrent if and only if either
  $1<\alpha\leq2$ and $\beta=0$ or $\alpha=1$. In this paper, in the  case  when $\alpha\neq1$ and $\beta=0$, we prove the same fact by using a different
technique.

In the general case, R.L. Schilling and J. Wang \cite[Theorem 1.1
(ii)]{rene-wang-feller}    have developed a Chung-Fuchs type
condition for transience for ``nice" Feller processes. By applying
this condition to the stable-like process $\process{X^{\alpha}}$,
they have
 shown that  $\limsup_{|x|\longrightarrow\infty}\alpha(x)<1$ is a sufficient condition for transience.
In this paper, we relax the assumption
$\limsup_{|x|\longrightarrow\infty}\alpha(x)<1$, that is,  we give a
sufficient condition for transience without any further assumptions
on the function $\alpha(x)$.

Next, J. Wang \cite{wang-ergodic} has given sufficient conditions
for recurrence and ergodicity of general one-dimensional L\'{e}vy
type Feller processes, that is, the Feller processes given by the
following infinitesimal generator
\begin{align*}\mathcal{A}f(x)=b(x)f'(x)+\frac{1}{2}a(x)f''(x)+\int_{\R}\left(f(y+x)-f(x)-f'(x)y1_{\{z:|z|\leq1\}}(y)\right)\nu(x,dy),\end{align*}
where  $a(x)\geq0$ and $b(x)\in\R$  are Borel measurable functions
and $\nu(x,\cdot)$ is a $\sigma$-finite Borel kernel on $\R\times
\mathcal{B}(\R)$, such that $\nu(x,\{0\})=0$ and $\int_{\R}(1\wedge
y^{2})\nu(x,dy)<\infty$ holds for all $x\in\R$. By applying these
conditions to the stable-like case, he has shown that
$$\textrm{if}\quad \alpha(x)>1\quad \textrm{and}\quad \frac{c(x)}{\alpha(x)-2}+x\beta(x)+\frac{2c(x)|x|^{2-\alpha(x)}}{\alpha(x)-1}\leq0$$  for all $|x|$ large enough,
then the corresponding stable-like process  is recurrent (see
\cite[Theorem 1.4 (i)]{wang-ergodic}).  Note that the above
recurrence condition does not cover the zero drift case and the case
when $\alpha(x)\leq1$. In particular, it does not cover a symmetric
stable L\'{e}vy process case. Further, conditions for ergodicity
presented in that paper do not cover the stable-like case (see
\cite[Theorem 1.4 (ii)]{wang-ergodic}).
  In this
paper, we give a sufficient condition for recurrence without any
further assumptions on the function $\alpha(x)$ and a sufficient
condition for ergodicity in the case when
$1<\inf\{\alpha(x):x\in\R\}$.

Furthermore, in the case when $\alpha(x)$ and $\gamma(x)$ are
periodic functions with the same period and when $\beta(x)=0,$ B.
Franke \cite{franke-periodic, franke-periodicerata} has shown that
the recurrence and transience property of the corresponding
stable-like process  depends only on  the minimum of the function
$\alpha(x)$, that is, if the set
$\{x\in\R:\alpha(x)=\alpha_0:=\inf_{x\in\R}\alpha(x)\}$ has positive
Lebesgue measure, then the corresponding stable-like process
 is recurrent if and only if $\alpha_0\geq 1.$

Finally, if the functions $\alpha(x)$, $\beta(x)$ and $\gamma(x)$
are of the form
\begin{align*}\alpha(x)=\left\{\begin{array}{cc}
                                                      \alpha, & x<-k \\
                                                      \beta, & x>k,
                                                    \end{array}\right.   \quad \beta(x)=0\quad\textrm{and}\quad  \gamma(x)=\left\{\begin{array}{cc}
                                                      \gamma, & x<-k \\
                                                      \delta,&
                                                      x>k,
                                                    \end{array}\right.\end{align*}
                                                    where
                                                    $\alpha,
                                                    \beta\in(0,2)$,
                                                    $\gamma,\delta\in(0,\infty)$
                                                    and $k>0$,  then by using an
                                                    overshoot
                                                    approach,
                                                    B. B\"ottcher
                                                    \cite{bjoern-overshoot} has shown that
                                                    the corresponding stable-like process
 is recurrent if and only if
$\alpha+\beta\geq2.$

For the Dirichlet forms approach to the problem of recurrence and
transience of stable-like processes without the drift term we refer
the reader to \cite{uemura1, uemura2}, and for the discrete-time
analogous of stable-like processes and their recurrence and
transience property we refer the reader to \cite{sandric-rectrans,
sandric-periodic}.

Now, let us state the main results of this paper.

\begin{theorem}\label{tm1.1}
\begin{enumerate}
  \item [(i)] Let
  $\liminf_{|x|\longrightarrow\infty}\alpha(x)\geq1.$
 If \begin{align}\label{eq:1.2}\limsup_{|x|\longrightarrow\infty}\left(\rm sgn(\it {x})\frac{\alpha(x)}{c(x)}|x|^{\alpha(x)-\rm{1}}\beta(x)+\pi\rm{ctg}\left(\frac{\pi\alpha(\it{x})}{\rm{2}}\right)\right)<\rm{0},\end{align}
 then the corresponding stable-like process $\process{X^{\alpha}}$
  is recurrent.

  \item [(ii)] If
\begin{align}\label{eq:1.3}\liminf_{|x|\longrightarrow\infty}\left(\rm
sgn(\it
{x})\frac{\alpha(x)}{c(x)}|x|^{\alpha(x)-\rm{1}}\beta(x)+\pi\rm{ctg}\left(\frac{\pi\alpha(\it{x})}{\rm{2}}\right)\right)>\rm{0},\end{align}
then the corresponding stable-like process $\process{X^{\alpha}}$
  is transient.
\end{enumerate}
\end{theorem}
The following  constant will appear in the statement of the
following theorem. For $\alpha\in(0,2)$ and $\theta\in(0,\alpha)$
let
$$E(\alpha,\theta):=\frac{\alpha}{\theta}\sum_{i=1}^{\infty}{\theta
\choose 2i}\frac{2}{2i-\alpha}-\frac{2}{\theta}+\frac{\alpha\,
_2F_1\left(-\theta,\alpha-\theta,1+\alpha-\theta;-1\right)+\alpha\,
_2F_1\left(-\theta,\alpha-\theta,1+\alpha-\theta;1\right)}{\theta(\alpha-\theta)},$$
where ${z \choose n}$ is the binomial coefficient and
$_2F_1\left(a,b,c;z\right)$ is the Gauss hypergeometric function
(see Section 3 for the definition of this function).
\begin{theorem}\label{tm1.2} Let $1<\alpha:=\inf\{\alpha(x):x\in\R\}.$ If
\begin{align}\label{eq:1.4}\limsup_{\theta\longrightarrow\alpha}\limsup_{|x|\longrightarrow\infty}\left(sgn(\it{x})\frac{\alpha(\it{x})}{c(\it{x})}|x|^{\alpha(x)-\rm{1}}\rm
\beta(\it{x})+\frac{\alpha(\it{x})}{\theta
c(\it{x})}|x|^{\alpha(x)-\theta}+E(\alpha(x),\theta)\right)<\rm{0},\end{align}
 then the corresponding stable-like process
$\process{X^{\alpha}}$ is ergodic.
\end{theorem}
Let us give several remarks about Theorems \ref{tm1.1} and
\ref{tm1.2}. Firstly, note that condition (\ref{eq:1.4}) implies
condition (\ref{eq:1.2}). To see this, note that
\begin{align*}\pi\rm{ctg}\left(\frac{\pi\alpha(\it{x})}{\rm{2}}\right)
<\frac{\alpha(\it{x})}{\theta
c(\it{x})}|\it{x}|^{\alpha(x)-\theta}+E(\alpha(x),\theta)\end{align*}
for all $0<\theta<\alpha$ and all $|x|$ large enough. Thus, as it is
commented in Section 2, ergodicity implies recurrence. Secondly, let
$0<\theta<\liminf_{|x|\longrightarrow\infty}\alpha(x)$ be arbitrary,
then
\begin{align}\label{eq:1.5}\limsup_{|x|\longrightarrow\infty}\left(sgn(\it{x})\frac{\alpha(\it{x})}{c(\it{x})}|x|^{\alpha(x)-\rm{1}}\rm
\beta(\it{x})+E(\alpha(x),\theta)\right)<\rm{0}\end{align} implies
recurrence of the corresponding stable-like process
 (see the proof of Theorem \ref{tm1.2}).
Further, it can be proved that the function $\theta\longmapsto
E(\alpha,\theta)$ is strictly increasing, thus we choose $\theta$
close to zero. Next, from (\ref{eq:3.2}), (\ref{eq:3.3}) and
(\ref{eq:3.4}), it is easy to see that
$$\lim_{\theta\longrightarrow0}E(\alpha,\theta)=\pi\rm{ctg}\left(\frac{\pi\alpha}{\rm{2}}\right).$$ Hence, (\ref{eq:1.5}) becomes (\ref{eq:1.2}).
Thirdly, note that if
$\limsup_{|x|\longrightarrow\infty}\alpha(x)<1$, than the
corresponding stable like process cannot be recurrent. Fourthly, the
assumption $\liminf_{|x|\longrightarrow\infty}\alpha(x)\geq1$ in
Theorem \ref{tm1.1} (i) is not restrictive, that is, we can state
Theorem \ref{tm1.1} (i) without any  assumptions about the function
$\alpha(x)$ (see the proof of Theorem \ref{tm1.1} (i)). But, if we
allow that $\liminf_{|x|\longrightarrow\infty}\alpha(x)<1$, then,
clearly, (\ref{eq:1.2}) does not hold. Finally, as a simple
consequence of Theorem \ref{tm1.1}  we get a new proof for
 the well-known recurrence and transience property of L\'{e}vy
 processes.
\begin{corollary}\label{c1.3}
A one-dimensional  stable L\'{e}vy process given by the symbol
(characteristic exponent) $p(\xi)=\gamma|\xi|^{\alpha},$ where
$\alpha\neq1$ and $\gamma\in(0,\infty)$, is recurrent if and only if
$\alpha>1$.
\end{corollary}
Note that Theorem \ref{tm1.1} (i) does not imply the recurrence
property of one-dimensional symmetric $1$-stable L\'{e}vy
 process since, in this case, the left-hand side in (\ref{eq:1.2})
 equals to zero.

 Now, we explain our strategy of  proving the  main results. The
proofs of Theorems~\ref{tm1.1} and~\ref{tm1.2} are based on the
\emph{Foster-Lyapunov  criteria}  (see   Theorem \ref{tm2.3}). These
criteria are based on finding an appropriate test function $V(x)$
(positive and unbounded in the recurrent case, positive and bounded
in the transient case and positive and finite in the ergodic case),
such that $\mathcal{A}^{\alpha}V(x)$ is well defined, and a compact
set $C\subseteq\R$, such that $\mathcal{A}^{\alpha}V(x)\leq0$ in the
recurrent case, $\mathcal{A}^{\alpha}V(x)\geq0$ in the transient
case and $\mathcal{A}^{\alpha}V(x)\leq-1$ in the ergodic case for
all $x\in C^{c}.$  The idea is to find test function $V(x)$ such
that the associated level sets $C_V(r):=\{y:V(y)\leq r\}$ are
compact sets and such that $C_V(r)\uparrow\R$, when
$r\nearrow\infty$, in the cases of recurrence and ergodicity and
$C_V(r)\uparrow\R$, when $r\nearrow1$, in the case of transience. In
the recurrent case, for the test function we take $V(x)=\ln(1+|x|)$.
In the transient case we take $V(x)=1-(1+|x|)^{-\theta}$, where
$0<\theta<1$ is arbitrary, and in the ergodic case we take
$V(x)=|x|^{\theta}$, where $1<\theta<\inf_{x\in\R}\alpha(x)$ is
arbitrary. Now, by proving
 that
$$\limsup_{|x|\longrightarrow\infty}\frac{\alpha(x)|x|^{\alpha(x)}}{c(x)}\mathcal{A}^{\alpha}V(x)<0$$
in the recurrent case,
$$\liminf_{|x|\longrightarrow\infty}\frac{\alpha(x)|x|^{\alpha(x)+\theta}}{c(x)}\mathcal{A}^{\alpha}V(x)>0$$
in the transient case and
$$\limsup_{|x|\longrightarrow\infty}\frac{\alpha(x)|x|^{\alpha(x)-\theta}}{c(x)}\left(\mathcal{A}^{\alpha}V(x)+1\right)<0$$
in the ergodic case,
 the proofs of Theorems~\ref{tm1.1}
and~\ref{tm1.2} are accomplished.

Let us remark that a similar approach,  using similar test
functions, can be found in \cite{lamperti}, \cite{rus} and
\cite{sandric-rectrans} in the discrete-time case and in
\cite{stramer} and \cite{wang-ergodic} in the continuous-time case.

The paper is organized as follows. In Section 2 we recall some
preliminary and auxiliary results regarding long-time behavior of
general Markov processes and we discuss several structural
properties of stable-like processes  which will be crucial in
finding sufficient conditions for  recurrence, transience and
ergodic properties. Further, we also give and discuss some
consequences of the main results. Finally, in Section 3, using the
Foster-Lyapunov  criteria, we give  proofs of Theorems \ref{tm1.1}
and \ref{tm1.2}.

Throughout the paper we use the following notation.  We write
$\ZZ_+$ and $\ZZ_-$ for nonnegative  and nonpositive integers,
respectively. By $\lambda(\cdot)$
                                                         we denote
                                                         the
                                                         Lebesgue
                                                         measure on the Borel $\sigma$-algebra $\mathcal{B}(\R^{d})$. Furthermore, $\{\Omega,\mathcal{F},\{\mathbb{P}^{x}\}_{x\in\R^{d}},\process{\mathcal{F}},\process{X}\}$, $\process{X}$ in the sequel, will denote an arbitrary Markov process on $\R^{d}$ with transition kernel $p^{t}(x,\cdot):=\mathbb{P}^{x}(X_t\in \cdot)$ and
                                                         $\process{X^{\alpha}}$  will denote
                                                         the
stable-like process given by the infinitesimal generator
(\ref{eq:1.1}).

\section{Preliminary and auxiliary results}

\quad \ \ In this section we recall some preliminary and auxiliary
results regarding long-time behavior of general Markov processes and
we discuss several structural properties of stable-like processes.
\begin{definition} Let $\process{X}$ be a c\`adl\`ag strong Markov
process on $\R^{d}$. The process $\process{X}$ is called
\begin{enumerate}
                      \item [(i)] \emph{Lebesgue irreducible} if
                      $\lambda(B)>0$ implies
                      $\int_{0}^{\infty}p^{t}(x,B)dt>0$ for all
                      $x\in\R^{d}$.
                      \item [(ii)] \emph{recurrent} if it is Lebesgue
                      irreducible and if $\lambda(B)>0$ implies $\int_{0}^{\infty}p^{t}(x,B)dt=\infty$ for all
                      $x\in\R^{d}$.
\item [(iii)] \emph{Harris recurrent} if it is Lebesgue
                      irreducible and if $\lambda(B)>0$ implies $\mathbb{P}^{x}(\tau_B<\infty)=1$ for all
                      $x\in\R^{d}$, where $\tau_B:=\inf\{t\geq0:X_t\in
                      B\}.$
                      \item [(iv)] \emph{transient} if it is Lebesgue
                      irreducible and if there exists a countable
                      covering of $\R^{d}$ with  sets
$\{B_j\}_{j\in\N}\subseteq\mathcal{B}(\R^{d})$, such that for each
$j\in\N$ there is a finite constant $M_j\geq0$ such that
$\int_0^{\infty}p^{t}(x,B_j)dt\leq M_j$ holds for all $x\in\R^{d}$.
                    \end{enumerate}
                                                         \end{definition}
 Note that the Lebesgue irreducibility of  stable L\'{e}vy processes  is trivially satisfied, and  the Lebesgue irreducibility of
 general stable-like processes has been shown in \cite[Theorem
 3.3]{rene-wang-feller}. Hence, according to \cite[Theorem
2.3]{tweedie-mproc}, every stable-like process is either recurrent
or transient. Further, every Harris recurrent process is recurrent
but in general, these two properties are not equivalent. They differ
on the set of the irreducibility measure zero (see \cite[Theorem
2.5]{tweedie-mproc}). In the case of the stable-like process
$\process{X^{\alpha}}$, by \cite[Theorems 4.2 and
4.4]{bjoern-overshoot}, these two properties are equivalent.

A $\sigma$-finite measure $\pi(\cdot)$ on $\mathcal{B}(\R^{d})$ is
called an \emph{invariant measure} for the Markov process
$\process{X}$ if
$$\pi(B)=\int_{\R^{d}}p^{t}(x,B)\pi(dx)$$ holds for all $t>0$ and all $B\in\mathcal{B}(\R^{d}).$ It is shown in \cite[Theorem 2.6]{tweedie-mproc} that if $\process{X}$ is
a recurrent process then there exists a unique (up to constant
multiples) invariant measure $\pi(\cdot)$. If the invariant measure
is finite, then it may be normalized to a probability measure.  If
$\process{X}$ is (Harris) recurrent with finite invariant measure
$\pi(\cdot)$, then $\process{X}$ is called \emph{positive (Harris)
recurrent}, otherwise it is called \emph{null (Harris) recurrent}.
The Markov process $\process{X}$ is called \emph{ergodic} if an
invariant probability measure $\pi(\cdot)$ exists and if
$$\lim_{t\longrightarrow\infty}||p^{t}(x,\cdot)-\pi(\cdot)||=0$$
hods for all $x\in\R^{d},$ where $||\cdot||$ denotes the total
variation norm on the space of signed measures. One would expect
that every positive (Harris) recurrent process  is ergodic, but in
general this is not true (see \cite{meyn-tweedie-II}). In the case
of the stable-like process $\process{X^{\alpha}}$,  these two
properties coincide. Indeed, according to \cite[Theorem
6.1]{meyn-tweedie-II} and \cite[Theorem 3.3]{rene-wang-feller} it
suffices to show that if
 $\process{X^{\alpha}}$ possess an invariant probability measure $\pi(\cdot)$, then
it is recurrent. Assume that $\process{X^{\alpha}}$ is transient.
Then there exist a countable
                      covering of $\R$ with  sets
$\{B_j\}_{j\in\N}\subseteq\mathcal{B}(\R)$, such that for each
$j\in\N$ there is a finite constant $M_j\geq0$ such that
$\int_0^{\infty}p^{t}(x,B_j)dt\leq M_j$ holds for all $x\in\R$. Let
$t>0$ be arbitrary. Then for each $j\in\N$ we have
$$t\pi(B_j)=\int_0^{t}\int_{\R}p^{s}(x,B_j)\pi(dx)ds\leq M_j.$$ By
letting $t\longrightarrow\infty$ we get that $\pi(B_j)=0$ for all
$j\in\N$, which is impossible. Let us remark that a stable L\'{e}vy
process is never ergodic since the Lebesgue measure satisfies the
invariance property.

Due to the above discussion and from Theorems \ref{tm1.1} (i),
\ref{tm1.2} and \cite[Theorem 3.2]{meyn-tweedie-II}, we get the
following two additional long-time properties of stable-like
processes.
\begin{corollary}\begin{enumerate}
                   \item [(i)] Under   assumptions of Theorem  \ref{tm1.1} (i),  for each initial position
$x\in\R$ and each   covering $\{O_n\}_{n\in\N}$ of $\R$ by open
bounded sets we have
$$\mathbb{P}^{x}\left(\bigcap_{n=1}^{\infty}\bigcup_{m=0}^{\infty}\left\{\int_m^{\infty}1_{\{X^{\alpha}_t\in
O_n\}}dt=0\right\}\right)=0,\quad x\in\R.$$ In other words, for each
initial position $x\in\R$ the event $\{X_t^{\alpha}\in C^{c}\
\textrm{for any compact set}\ C\subseteq\R\ \textrm{and all}\
t\in\R_+\ \textrm{sufficiently large}\}$ has probability $0$.
                   \item [(ii)]Under   assumptions of Theorem \ref{tm1.2},  for each initial position $x\in\R$ and each $\varepsilon>0$,
there exists a compact set $C\subseteq\R$ such that
$$\liminf_{
t\longrightarrow\infty}\mathbb{E}^{x}\left[\frac{1}{t}\int_0^{t}1_{\{X^{\alpha}_s\in
C\}}ds\right]\geq 1-\varepsilon.$$
                 \end{enumerate}
\end{corollary}

As already mentioned, the proofs of Theorems~\ref{tm1.1} and
\ref{tm1.2} are based on the Foster-Lyapunov  criteria. Let us
recall several notions regarding Markov processes we are going to
need in the sequel. The \emph{extended domain} of a c\`adl\`ag
Markov process $\process{X}$ on $\R^{d}$   is defined by
\begin{align*}\mathcal{D}(\tilde{\mathcal{A}}):=\Bigg\{&f(x)\
\textrm{is Borel measurable}:\ \textrm{there is a Borel measurable
function}\ g(x)\ \textrm{such that}\\
&f(X_t)-f(X_0)-\int_{0}^{t}g(X_s)ds\ \textrm{is a local martingale
under}\ \{\mathbb{P}^{x}\}_{x\in\R^{d}}\Bigg\}.\end{align*} Let us
remark that in general, the function $g(x)$ does not have to be
unique (see \cite[Page 24]{ethier-kurtz-book})).  For
$f\in\mathcal{D}(\tilde{\mathcal{A}})$ we define
$$\tilde{\mathcal{A}}f=\left\{g(x)\ \textrm{is Borel measurable}:
f(X_t)-f(X_0)-\int_{0}^{t}g(X_s)ds\ \textrm{is a local martingale
under}\ \{\mathbb{P}^{x}\}_{x\in\R^{d}}\right\}.$$ We  call
$\tilde{\mathcal{A}}$ the \emph{extended generator} of
$\process{X}$. A function $g\in\tilde{\mathcal{A}}f$ is usually
abbreviated by $\tilde{\mathcal{A}}f(x):=g(x)$.
 Clearly, if
$\mathcal{A}$ is the infinitesimal generator of the Markov process
$\process{X}$ with the domain $\mathcal{D}(\mathcal{A})$, then
$\mathcal{D}(\mathcal{A})\subseteq\mathcal{D}(\tilde{\mathcal{A}})$
and for $f\in\mathcal{D}(\mathcal{A})$ the function
$\mathcal{A}f(x)$ is contained in $\tilde{\mathcal{A}}f$ (see
\cite[Proposition IV.1.7]{ethier-kurtz-book}). In the case of the
stable-like process $\process{X^{\alpha}}$, obviously we have
\begin{align}\label{eq:2.1}\left\{f\in
C^{2}(\R):\left|\int_{\{|y|\geq1\}}(f(y+x)-f(x))\frac{c(x)}{|y|^{\alpha(x)+1}}dy\right|<\infty\
\textrm{for all}\
x\in\R\right\}\subseteq\mathcal{D}(\tilde{\mathcal{A}}),\end{align}
and for the function $\tilde{\mathcal{A}}f(x)$ we can take exactly
the function $\mathcal{A}^{\alpha}f(x),$ where
$\mathcal{A}^{\alpha}$ is the infinitesimal generator of the
stable-like process $\process{X^{\alpha}}$ given by (\ref{eq:1.1}).

Next, let $\process{X}$ be a c\`adl\`ag Markov process  on $\R^{d}$
and let
$$T_c:=\inf\{t\geq0: X_t\notin\R^{d}\}\quad\textrm{and}\quad T_e:=\lim_{n\longrightarrow\infty}\inf\{t\geq0:
|X_t|>n\}.$$  The process $\process{X}$ is called
\emph{conservative}  if $\mathbb{P}^{x}(T_c=\infty)=1$ for all
$x\in\R^{d}$ and \emph{non-explosive} if
$\mathbb{P}^{x}(T_e=\infty)=1$ for all $x\in\R^{d}$. By
\cite[Theorem 5.2]{rene-conserv}, the stable-like process
$\process{X^{\alpha}}$ is always conservative and then, sice it has
c\`adl\`ag paths, it also non-explosive.

A function $V:\R^{d}\longrightarrow\R_+$ is called a \emph{norm-like
function} (for a c\`adl\`ag Markov process $\process{X}$) if
$V\in\mathcal{D}(\tilde{\mathcal{A}})$ and the level sets
$\{x:V(x)\leq r\}$ are precompact sets for each level $r\geq0$.

Finally, a set  $C\in \mathcal{B}(\R^{d})$ is called
$\nu_a$\emph{-petite set} (for a c\`adl\`ag Markov process
$\process{X}$) if there exist a probability measure $a(\cdot)$
  on $\mathcal{B}(\R_+)$ and a nontrivial measure $\nu_a(\cdot)$ on $\mathcal{B}(\R^{d})$ such that
$$\int_{0}^{\infty}p^{t}(x,B)a(dt)\geq \nu_a(B)$$ holds for all $x\in C$ and all $B\in
\mathcal{B}(\R^{d})$.

Now, we state the Foster-Lyapunov  criteria (see
 \cite[Theorems 3.2 and 4.2]{meyn-tweedie-III} and \cite[Theorem
 3.3]{stramer-tweedie-stability}).
\begin{theorem}\label{tm2.3}  Let $\process{X}$ be a non-explosive
Lebesgue irreducible  c\`adl\`ag Markov process on $\R^{d}$.
\begin{enumerate}
  \item [(i)] If  every compact set is a petite set and if there exist a
  compact
set $C\subseteq\R^{d}$ with $\lambda(C)>0$, a constant $d>0$ and a
norm-like function $V(x)$, such that  $$\tilde{\mathcal{A}}V(x)\leq
d1_C(x),\quad x\in\R^{d},$$ then the  process $\process{X}$ is
Harris recurrent.
\item [(ii)]If there exist a bounded Borel
measurable function
  $V:\R^{d}\longrightarrow\R_+$ and  closed sets
  $C,D\subseteq\R^{d}$, $D\subseteq C^{c}$,
  such that\begin{enumerate}
\item [(a)] $\lambda(C)>0$, $\lambda(D)>0$ and $\sup_{x\in
C}V(x)<\inf_{x\in D}V(x)$
        \item [(b)] $\tilde{\mathcal{A}}V(x)\geq 0\quad\textrm{for all}\ x\in
        C^{c}$,
            \end{enumerate}
            then the  process $\process{X}$ is transient.
  \item [(iii)]If every compact set is a petite set and if  there exist  $d>0$, a compact set
$C\subseteq\R^{d}$ with $\lambda(C)>0$, a Borel measurable function
$f(x)\geq1$ and a nonnegative function
$V\in\mathcal{D}(\tilde{\mathcal{A}})$, such that
\begin{enumerate}
  \item [(a)] $V(x)$ is bounded on $C$
  \item [(b)] $\tilde{\mathcal{A}}V(x)\leq-f(x)+
d1_C(x)\quad\textrm{for all}\  x\in\R^{d},$
\end{enumerate}
then the  process $\process{X}$ is positive Harris recurrent and
$$\pi(f):=\int_{\R^{d}}f(x)\pi(dx)<\infty,$$ where $\pi(\cdot)$ is the
invariant measure for $\process{X}$.
\end{enumerate}
\end{theorem}
Let us remark that in the case of the stable-like process
$\process{X^{\alpha}}$, according to \cite[Theorems 5.1 and
7.1]{tweedie-mproc}, the first requirements of  Theorem \ref{tm2.3}
(i) and (iii) always hold, that is, every compact set is a petite
set.

We end this section with the following observation. Assume that
$\process{X}$ is an ergodic Markov process with invariant measure
$\pi(\cdot)$. Then, clearly,
$$\lim_{t\longrightarrow\infty}\mathbb{E}^{x}[f(X_t)]=\pi(f)$$
holds for all $x\in\R^{d}$ and all bounded Borel measurable
functions $f(x).$ In what follows,  we extend this convergence to a
wider class of functions. For any Borel measurable function
$f(x)\geq 1$ and any signed measure $\mu(\cdot)$ on
$\mathcal{B}(\R^{d})$ we write
$$||\mu||_f:=\sup_{|g|\leq f}|\mu(g)|,$$ where the supremum is taken over all Borel measurable functions $g:\R^{d}\longrightarrow\R$ satisfying $|g(x)|\leq f(x)$ for all $x\in\R^{d}$.  A Markov process $\process{X}$ is called $f$-\emph{ergodic} if it is positive Harris
recurrent with invariant probability measure $\pi(\cdot)$, if
$\pi(f) < \infty$, and
$$\lim_{t\longrightarrow\infty}||p^{t}(x,\cdot)-\mu(f)||_f=0,\quad x\in\R^{d}.$$
Note that $||\cdot||_1=||\cdot||$. Hence,   $f$-ergodicity implies
ergodicity. Now, by Theorems  \ref{tm1.2}, \ref{tm2.3} (iii) and
\cite[Theorem 5.3 (ii)]{meyn-tweedie-III}, we get the following
sufficient condition for $f$-ergodicity.
\begin{theorem}
\label{tm2.4}Let $1<\alpha:=\inf_{x\in\R}\alpha(x)$ and let
$\theta\in(1,\alpha)$ be arbitrary.
 If there exist Borel measurable function $f(x)\geq1$ and strictly
 increasing, nonnegative and convex function $\phi(x)$, such that
 $|x|^{\theta}=\phi(f(x))$ for all $|x|$ large enough and
\begin{align}\label{eq:2.2}\limsup_{|x|\longrightarrow\infty}\left(\rm
sgn(\it{x})\frac{\alpha(\it{x})}{c(\it{x})}|x|^{\alpha(x)-\rm{1}}\beta(\it{x})+\frac{\alpha(\it{x})}{\theta
c(\it{x})}|x|^{\alpha(x)-\theta}f(x)+E(\alpha(x),\theta)\right)<\rm{0},\end{align}
 then the corresponding stable-like process $\process{X^{\alpha}}$ is
$f$-ergodic (recall that the constant $E(\alpha,\theta)$ is defined
in  Theorem \ref{tm1.2}).
\end{theorem}
Now, let $\eta\in(0,\theta]$ be arbitrary. By taking
$f(x)=|x|^{\eta}$ and $\phi(x)=|x|^{\frac{\theta}{\eta}}$ for all
$|x|$ large enough, we get the following corollary.
\begin{corollary}
Let $1<\alpha:=\inf_{x\in\R}\alpha(x)$. If
\begin{align*}\limsup_{\theta\longrightarrow\alpha}\limsup_{|x|\longrightarrow\infty}\left(\rm
sgn(\it{x})\frac{\alpha(\it{x})}{c(\it{x})}|x|^{\alpha(x)-\rm{1}}\beta(\it{x})+\frac{\alpha(\it{x})}{\theta
c(\it{x})}|x|^{\alpha(x)-\theta+\eta}+E(\alpha(x),\theta)\right)<\rm{0}\end{align*}
for some $\eta\in(0,\alpha)$, then the corresponding stable-like
process $\process{X^{\alpha}}$ is $f$-ergodic  for every function
$f(x)\geq1$ such that $f(x)\leq|x|^{\eta}$ for all $|x|$ large
enough.
\end{corollary}

\section{Proof of  the main results}
\quad \ \ In this section we give  proofs of Theorems \ref{tm1.1}
and \ref{tm1.2}. Before the proofs, we recall several special
functions we need. The Gamma function is defined by the formula
$$\Gamma(z):=\int_0^{\infty}t^{z-1}e^{-t}dt,\quad z\in\CC,\ \mathrm{Re}\it(z)>\rm
0.$$ It can be analytically continued on $\CC\setminus\ZZ_-$  and it
satisfies the following two well-known properties
\be\label{eq:3.1}\Gamma(z+1)=z\Gamma(z)\quad\textrm{and}\quad\Gamma(1-z)\Gamma(z)=\frac{\pi}{\sin(\pi
z)}.\ee

The Digamma function is a function defined by
$\Psi(z):=\frac{\Gamma'(z)}{\Gamma(z)},$ for $z\in
\CC\setminus\ZZ_-,$ and it satisfies the following properties:
\begin{enumerate}
\item [(i)]
  \be\label{eq:3.2}\Psi(1+z)=-\gamma+\sum_{n=1}^{\infty}\frac{z}{n(n+z)},\ee
  where $\gamma$ is the Euler's number
\item [(ii)]\be\label{eq:3.3}\Psi(1+z)=\Psi(z)+\frac{1}{z}\ee
  \item [(iii)]
\be\label{eq:3.4}\Psi(1-z)=\Psi(z)+\pi\rm{ctg}\,(\pi z).\ee
\end{enumerate}

The Gauss hypergeometric function is defined by the formula
\be\label{eq:3.5}_2F_1(a,b,c;z):=\sum_{n=0}^{\infty}\frac{(a)_n(b)_n}{(c)_n}\frac{z^{n}}{n!},\ee
for $a,b,c,z\in \CC$, $c\notin\ZZ_-$, where for $w\in\CC$ and
$n\in\ZZ_+$, $(w)_n$ is defined by
$$(w)_0=1\quad \textrm{and}\quad (w)_n=w(w+1)\cdots(w+n-1).$$
The series (\ref{eq:3.5}) absolutely converges on $|z|<1$,
absolutely converges on $|z|\leq1$ when $\mathrm{Re}\it(c-a-b)>\rm
0$, conditionally converges on $|z|\leq1$, except for $z=1$, when
$-1<\mathrm{Re} \it(c-b-a)\leq \rm0$ and diverges when
$\mathrm{Re}\it (c-b-a)\leq\rm-1$. In the case when $\mathrm{Re}
\it(c)>\mathrm{Re} \it (b)>\rm0$, it can be analytically continued
on $\CC\setminus(1,\infty)$ by the formula
\be\label{eq:3.6}_2F_1(a,b,c;z)=\frac{\Gamma(c)}{\Gamma(b)\Gamma(c-b)}\int_{0}^{1}t^{b-1}(1-t)^{c-b-1}(1-tz)^{-a}dt,\ee
and for $a,b\in\CC,$ $c\in\CC\setminus\ZZ_-$ and
$z\in\CC\setminus(0,\infty)$ it satisfies the following relation
\begin{align}\label{eq:3.7}
                              _2F_1(a,b,c;z)=&\frac{\Gamma(c)\Gamma(b-a)}{\Gamma(b)\Gamma(c-a)}(-z)^{-a}\,_2F_1\left(a,1-c+a,1-b+a;\frac{1}{z}\right)\nonumber\\&+\frac{\Gamma(c)\Gamma(a-b)}{\Gamma(a)\Gamma(c-b)}(-z)^{-b}\,_2F_1\left(b,1-c+b,1-a+b;\frac{1}{z}\right).\end{align}
For further properties of the Gamma function, the Digamma function
and hypergeometric functions see \cite[Chapters 6 and
15]{abram-stegun-book}. We also need the following  two lemmas.
\begin{lemma}\label{lm3.1}
$$\lim_{x\longrightarrow\infty}\frac{1}{x}\sum_{n=1}^{\infty}\frac{1}{n}\left(\frac{x}{1+x}\right)^{n}=0.$$
\end{lemma}
\begin{proof}
First, note that for $x\geq0$, by the binomial expansion of
$(1+x)^{n}$, we have
$$(1+x)^{n}\geq n x^{n-1}\quad\textrm{for all}\ n\in\N.$$ Now, the desired result follows by
the dominated convergence theorem.
\end{proof}
\begin{lemma}\label{lm3.2} Let $\alpha:\R\longrightarrow(0,1)\cup(1,2)$ be an
arbitrary function. Then for every $R\geq0$ we have
$$\lim_{x\longrightarrow\infty}\frac{1}{1-\alpha(x)}\left(1-\left(\frac{x}{x+R}\right)^{1-\alpha(x)}\right)=0.$$
\end{lemma}
\begin{proof}  Let
$0<\varepsilon<1$ be arbitrary. Since
$$\frac{1}{t}\left(1-\left(1-\varepsilon\right)^{t}\right)\leq-2\ln(1-\varepsilon)$$
holds for all $t\in(-1,0)\cup(0,1)$, we have
\begin{align*}0\leq\limsup_{x\longrightarrow\infty}\frac{1}{1-\alpha(x)}\left(1-\left(\frac{x}{x+R}\right)^{1-\alpha(x)}\right)&\leq\limsup_{x\longrightarrow\infty}\frac{1-\left(1-\varepsilon\right)^{1-\alpha(x)}}{1-\alpha(x)}\leq
                                                      -2\ln(1-\varepsilon).\end{align*}
                                                      Now, by letting
                                                      $\varepsilon\longrightarrow0$,
                                                      we have the
                                                      claim.
\end{proof}
\begin{proof}[\textbf{Proof of  Theorem} \ref{tm1.1} (i)]
                                                       The proof is divided in four steps.

\textbf{Step 1.} In the first step we explain our strategy of the
proof.  Let $\varphi\in C^{2}(\R)$ be an arbitrary nonnegative
function such that $\varphi(x)=|x|$, for $|x|>1$, and
$\varphi(x)\leq|x|$, for $|x|\leq1.$ Now, let us define the function
$V:\R\longrightarrow\R_+$ by the formula
$$V(x):=\ln(1+\varphi(x)).$$
Clearly, $V\in C^{2}(\R)$ and the level set  $C_V(r):=\{x:V(x)\leq
r\}$ is a compact set for all levels $r\geq0.$ Furthermore, it is
easy to see that
$$\left|\int_{\{|y|>1\}}\left(V(x+y)-V(x)\right)\frac{c(x)}{|y|^{\alpha(x)+1}}dy\right|<\infty$$
holds for all $x\in\R$. Hence, by the relation (\ref{eq:2.1}),
$V\in\mathcal{D}(\tilde{\mathcal{A}})$ and for the function
$\tilde{\mathcal{A}}V(x)$ we can take the function
$\mathcal{A}^{\alpha}V(x)$, where
                                                    $\mathcal{A}^{\alpha}$ is
                                                    the
                                                    infinitesimal
                                                    generator of the
                                                    stable-like
                                                    process
                                                    $\process{X^{\alpha}}$
                                                    given by (\ref{eq:1.1}).
In the sequel we  show that there exists $r_0>0$, large enough, such
that $\tilde{\mathcal{A}}V(x)\leq0$ for all $x\in
\left(C_V(r_0)\right)^{c}.$ Clearly, $\sup_{x\in
C_V(r_0)}|\tilde{\mathcal{A}}V(x)|<\infty.$ Thus, the desired result
follows from Theorem \ref{tm2.3} (i). In order to see this, since
$C_V(r)\uparrow\R$, when $r\nearrow\infty$, it suffices to show that
\begin{align}\label{eq:3.8}\displaystyle\limsup_{|x|\longrightarrow\infty}\frac{\alpha(x)}{c(x)}(1+|x|)^{\alpha(x)}\tilde{\mathcal{A}}V(x)<0.\end{align}

\textbf{Step 2.} In the second step, we find  more appropriate
expression for the function $\tilde{\mathcal{A}}V(x).$ We have
\begin{align*}\tilde{\mathcal{A}}V(x)=\mathcal{A}^{\alpha}V(x)=&\beta(x)V'(x)+\int_{\R}\left(V(x+y)-V(x)-V'(x)y1_{\{|y|\leq1\}}(y)\right)\frac{c(x)}{|y|^{\alpha(x)+1}}dy\nonumber\\
=&\beta(x)V'(x)+\int_{\{|y|\leq1\}}\left(V(x+y)-V(x)-V'(x)y\right)\frac{c(x)}{|y|^{\alpha(x)+1}}dy\\&+\int_{\{|y|>1\}}\left(V(x+y)-V(x)\right)\frac{c(x)}{|y|^{\alpha(x)+1}}dy.
\end{align*}
Let us define \begin{align*} A(x):&=\beta(x)V'(x)\\
B(x):&=\int_{\{|y|\leq1\}}\left(V(x+y)-V(x)-V'(x)y\right)\frac{c(x)}{|y|^{\alpha(x)+1}}dy\\
C(x):&=\int_{\{|y|>1\}}\left(V(x+y)-V(x)\right)\frac{c(x)}{|y|^{\alpha(x)+1}}dy.\end{align*}
Hence, in order to prove (\ref{eq:3.8}) it suffices to prove
\begin{align}\label{eq:3.9}&\limsup_{|x|\longrightarrow\infty}\frac{\alpha(x)}{c(x)}(1+|x|)^{\alpha(x)}\tilde{\mathcal{A}}V(x)=\limsup_{|x|\longrightarrow\infty}\frac{\alpha(x)}{c(x)}(1+|x|)^{\alpha(x)}(A(x)+B(x)+C(x))\nonumber\\
&\leq\limsup_{|x|\longrightarrow\infty}\frac{\alpha(\it{x})}{c(\it{x})}(1+|x|)^{\alpha(x)}\left(\rm
\it{A}(\it{x})+\it{C}(\it{x})\right)+\limsup_{|\it{x}|\longrightarrow\infty}\frac{\alpha(\it{x})}{c(x)}(\rm{1}+|\it{x}|)^{\alpha(\it{x})}B(\it{x})<\rm{0}.\end{align}
Furthermore, for  $x>0$ large enough we have
\begin{align*}
A(x)&=\frac{\beta(x)}{1+x},\\
B(x)&=\int_{\{|y|\leq1\}}\left(\ln(1+x+y)-\ln(1+x)-\frac{y}{1+x}\right)\frac{c(x)}{|y|^{\alpha(x)+1}}dy\\
&=\int_{-1}^{1}\left(\ln\left(1+\frac{y}{1+x}\right)-\frac{y}{1+x}\right)\frac{c(x)}{|y|^{\alpha(x)+1}}dy
\end{align*}
and

\begin{align*}C(x)=&\int_{\{y\leq-x-1\}}\left(\ln(1-x-y)-\ln(1+x)\right)\frac{c(x)}{|y|^{\alpha(x)+1}}dy\\
&+\int_{\{-x-1\leq y\leq-x+1\}}\left(\ln(1+\varphi(x+y))-\ln(1+x)\right)\frac{c(x)}{|y|^{\alpha(x)+1}}dy\\
&+\int_{\{-x+1\leq
y<-1\}}\left(\ln(1+x+y)-\ln(1+x)\right)\frac{c(x)}{|y|^{\alpha(x)+1}}dy\\
&+\int_{\{y>1\}}\left(\ln(1+x+y)-\ln(1+x)\right)\frac{c(x)}{y^{\alpha(x)+1}}dy\\
=&\int_{1+x}^{\infty}\ln\left(\frac{1-x+y}{1+x}\right)\frac{c(x)}{y^{\alpha(x)+1}}dy\\
&+\int_{x-1}^{1+x}\ln\left(\frac{1+\varphi(x-y)}{1+x}\right)\frac{c(x)}{y^{\alpha(x)+1}}dy\\
&+\int_{1}^{x-1}\ln\left(1-\frac{y}{1+x}\right)\frac{c(x)}{y^{\alpha(x)+1}}dy\\
&+\int_{1}^{\infty}\ln\left(1+\frac{y}{1+x}\right)\frac{c(x)}{y^{\alpha(x)+1}}dy.\end{align*}

\textbf{Step 3.} In the third step, we compute
$\limsup_{x\longrightarrow\infty}\frac{\alpha(x)}{c(x)}(1+x)^{\alpha(x)}B(x).$
By using the elementary inequality $t-t^{2}\leq\ln(1+t)\leq t$, we
get
$$-\frac{2c(x)}{(2-\alpha(x))(1+x)^{2}}=-\frac{c(x)}{(1+x)^{2}}\int_{-1}^{1}|y|^{1-\alpha(x)}dy\leq
B(x)\leq0.$$ Hence,
\begin{align}\label{eq:3.10}\lim_{x\longrightarrow\infty}\frac{\alpha(x)}{c(x)}(1+x)^{\alpha(x)}B(x)=0.\end{align}

\textbf{Step 4.} In the fourth step, we compute
$\limsup_{x\longrightarrow\infty}\frac{\alpha(\it{x})}{c(\it{x})}(1+x)^{\alpha(x)}\left(\it{A}(\it{x})+\it{C}(\it{x})\right).$
In the case when $\alpha(x)=1$,  we have
\begin{align*}
\frac{1+x}{c(x)}C(x)=&\ln\left(\frac{1}{1+x}\right)+\frac{\ln\left(\frac{1+x}{4}\right)}{x-1}+\frac{x}{x-1}\ln(1+x)\\
&+(1+x)\int_{x-1}^{1+x}\ln\left(\frac{1+\varphi(x-y)}{1+x}\right)\frac{dy}{y^{2}}-\ln(x-1)+x\ln(x)\\
&-\frac{\ln\left(\frac{4}{1+x}\right)}{x-1}+\frac{\ln(1+x)}{x-1}+\frac{x}{x-1}\ln(1+x)-\frac{x^{2}}{x-1}\ln(1+x)\\
&+\ln(2+x)+(1+x)\ln\left(1+\frac{1}{1+x}\right).\end{align*} Now,
since
$$\frac{2}{x-1}\ln\left(\frac{1}{x-1}\right)\leq(1+x)\int_{x-1}^{1+x}\ln\left(\frac{1+\varphi(x-y)}{1+x}\right)\frac{dy}{y^{2}}\leq\frac{2}{x-1}\ln\left(\frac{2}{x-1}\right),$$
by elementary computation we get
\be\label{eq:3.11}\lim_{x\longrightarrow\infty}\frac{1+x}{c(x)}C(x)=0.\ee
Further, in the case when $\alpha(x)\neq1$, using integration by
parts formula and (\ref{eq:3.6}),   we get
\begin{align*}
&\frac{\alpha(x)}{c(x)}(1+x)^{\alpha(x)}C(x)\\
&=\ln\left(\frac{2}{1+x}\right)+\frac{\,_2F_1\left(1,\alpha(x),1+\alpha(x);\frac{x-1}{1+x}\right)}{\alpha(x)}\\
&\ \ \ +\alpha(x)(1+x)^{\alpha(x)}\int_{x-1}^{1+x}\ln\left(\frac{1+\varphi(x-y)}{1+x}\right)\frac{dy}{y^{\alpha(x)+1}}\\
&\ \ \ +(1+x)^{\alpha(x)}\ln\left(1-\frac{1}{1+x}\right)-(1+x)^{\alpha(x)}(x-1)^{-\alpha(x)}\ln\left(\frac{2}{1+x}\right)\\
&\ \ \ +\frac{\,_2F_1\left(1,1-\alpha(x),2-\alpha(x);\frac{x-1}{1+x}\right)}{(\alpha(x)-1)(1+x)^{1-\alpha(x)}(x-1)^{\alpha(x)-1}}-\frac{\,_2F_1\left(1,1-\alpha(x),2-\alpha(x);\frac{1}{1+x}\right)}{(\alpha(x)-1)(1+x)^{1-\alpha(x)}}\\
&\ \ \
+(1+x)^{\alpha(x)}\ln\left(1+\frac{1}{1+x}\right)+\frac{\,_2F_1\left(1,\alpha(x),1+\alpha(x);-x-1\right)}{\alpha(x)(1+x)^{-\alpha(x)}}.\end{align*}
Let us put
\begin{align*}
C_1(x):=&\ln\left(\frac{2}{1+x}\right)-(1+x)^{\alpha(x)}(x-1)^{-\alpha(x)}\ln\left(\frac{2}{1+x}\right)\\
C_2(x):=&(1+x)^{\alpha(x)}\ln\left(1-\frac{1}{1+x}\right)+(1+x)^{\alpha(x)}\ln\left(1+\frac{1}{1+x}\right)\\
C_3(x):=&\alpha(x)(1+x)^{\alpha(x)}\int_{x-1}^{1+x}\ln\left(\frac{1+\varphi(x-y)}{1+x}\right)\frac{dy}{y^{\alpha(x)+1}}\\
C_4(x):=&\frac{\,_2F_1\left(1,\alpha(x),1+\alpha(x);-x-1\right)}{\alpha(x)(1+x)^{-\alpha(x)}}-\frac{\,_2F_1\left(1,1-\alpha(x),2-\alpha(x);\frac{1}{1+x}\right)}{(\alpha(x)-1)(1+x)^{1-\alpha(x)}}\\
C_5(x):=&\frac{\,_2F_1\left(1,1-\alpha(x),2-\alpha(x);\frac{x-1}{1+x}\right)}{(\alpha(x)-1)(1+x)^{1-\alpha(x)}(x-1)^{\alpha(x)-1}}+\frac{\,_2F_1\left(1,\alpha(x),1+\alpha(x);\frac{x-1}{1+x}\right)}{\alpha(x)}.
\end{align*}
Thus,
\begin{align}\label{eq:3.12}\frac{\alpha(x)}{c(x)}(1+x)^{\alpha(x)}C(x)=C_1(x)+C_2(x)+C_3(x)+C_4(x)+C_5(x).\end{align}
 Now, by applying Lemma \ref{lm3.2}, it is easy to see that
\begin{align}\label{eq:3.13}\lim_{x\longrightarrow\infty}C_1(x)=0,\end{align}
and
\begin{align}\label{eq:3.14}\lim_{x\longrightarrow\infty}C_2(x)=0.\end{align}
Further, since $0\leq\varphi(y)\leq1$, for $|y|\leq1$, we have
$$\ln\left(\frac{1}{1+x}\right)\left(\left(\frac{1+x}{x-1}\right)^{\alpha(x)}-1\right)\leq C_3(x)\leq\ln\left(\frac{2}{1+x}\right)\left(\left(\frac{1+x}{x-1}\right)^{\alpha(x)}-1\right).$$
Again by Lemma \ref{lm3.2}, it follows that
\begin{align}\label{eq:3.15}\lim_{x\longrightarrow\infty}C_3(x)=0.\end{align}
Further, by applying  (\ref{eq:3.1}) and (\ref{eq:3.7})  we get
 \begin{align}\label{eq:3.16}C_4(x)=&\frac{\,_2F_1\left(1,1-\alpha(x),2-\alpha(x);-\frac{1}{1+x}\right)}{(\alpha(x)-1)(1+x)^{1-\alpha(x)}}+\frac{\pi}{\sin(\alpha(x)\pi)}\nonumber\\&-\frac{\,_2F_1\left(1,1-\alpha(x),2-\alpha(x);\frac{1}{1+x}\right)}{(\alpha(x)-1)(1+x)^{1-\alpha(x)}}.
\end{align}
Now, since
$0<\inf\{\alpha(x):x\in\R\}\leq\sup\{\alpha(x):x\in\R\}<2$, by
(\ref{eq:3.5}), we have
\begin{align}\label{eq:3.17}\lim_{x\longrightarrow\infty}\left(\frac{\,_2F_1\left(1,1-\alpha(x),2-\alpha(x);-\frac{1}{x+1}\right)}{(\alpha(x)-1)(x+1)^{1-\alpha(x)}}-\frac{\,_2F_1\left(1,1-\alpha(x),2-\alpha(x);\frac{1}{x+1}\right)}{(\alpha(x)-1)(x+1)^{1-\alpha(x)}}\right)=0.\end{align}
Next,
\begin{align}\label{eq:3.18}C_5(x)=&\frac{\,_2F_1\left(1,1-\alpha(x),2-\alpha(x);\frac{x-1}{x+1}\right)}{(\alpha(x)-1)(x+1)^{1-\alpha(x)}(x-1)^{\alpha(x)-1}}-\frac{\,_2F_1\left(1,1-\alpha(x),2-\alpha(x);\frac{x-1}{x+1}\right)}{\alpha(x)-1}\nonumber\\
&+\frac{\,_2F_1\left(1,1-\alpha(x),2-\alpha(x);\frac{x-1}{x+1}\right)}{\alpha(x)-1}
+\frac{\,_2F_1\left(1,\alpha(x),1+\alpha(x);\frac{x-1}{x+1}\right)}{\alpha(x)}.
\end{align}
Again, since
$0<\inf\{\alpha(x):x\in\R\}\leq\sup\{\alpha(x):x\in\R\}<2$, by
applying (\ref{eq:3.5}) and Lemma \ref{lm3.1}, we get
\begin{align}\label{eq:3.19}\lim_{x\longrightarrow\infty}\left(\frac{\,_2F_1\left(1-\alpha(x),1,2-\alpha(x);\frac{x-1}{x+1}\right)}{(\alpha(x)-1)(x+1)^{1-\alpha(x)}(x-1)^{\alpha(x)-1}}-\frac{\,_2F_1\left(1-\alpha(x),1,2-\alpha(x);\frac{x-1}{x+1}\right)}{\alpha(x)-1}\right)=0,\end{align}
and from (\ref{eq:3.5}) we get
\begin{align*}&\frac{\,_2F_1\left(1-\alpha(x),1,2-\alpha(x);\frac{x-1}{x+1}\right)}{\alpha(x)-1}
+\frac{\,_2F_1\left(1,\alpha(x),1+\alpha(x);\frac{x-1}{x+1}\right)}{\alpha(x)}\\
&=\frac{1}{\alpha(x)-1}\sum_{n=0}^{\infty}\frac{1-\alpha(x)}{1-\alpha(x)+n}\left(\frac{x-1}{x+1}\right)^{n}+\frac{1}{\alpha(x)}\sum_{n=0}^{\infty}\frac{\alpha(x)}{\alpha(x)+n}\left(\frac{x-1}{x+1}\right)^{n}\\
&=\frac{1}{\alpha(x)-1}+\frac{1}{\alpha(x)}+\sum_{n=1}^{\infty}\frac{1-2\alpha(x)}{(\alpha(x)+n)(1-\alpha(x)+n)}\left(\frac{x-1}{x+1}\right)^{n}\\
&=\frac{1}{\alpha(x)-1}+\frac{1}{\alpha(x)}+\sum_{n=1}^{\infty}\frac{1-2\alpha(x)}{(\alpha(x)+n)(1-\alpha(x)+n)}\\
&\ \ \
+\sum_{n=1}^{\infty}\frac{1-2\alpha(x)}{(\alpha(x)+n)(1-\alpha(x)+n)}\left(\left(\frac{x-1}{x+1}\right)^{n}-1\right).\end{align*}
Clearly, by the dominated convergence theorem we have
\begin{align}\label{eq:3.20}\lim_{x\longrightarrow\infty}\sum_{n=1}^{\infty}\frac{1-2\alpha(x)}{(\alpha(x)+n)(1-\alpha(x)+n)}\left(\left(\frac{x-1}{x+1}\right)^{n}-1\right)=0,\end{align}
and by using the Taylor series expansion of the function $\rm ctg
(\it{y})$ we get
\begin{align}\label{eq:3.21}\frac{1}{\alpha(x)-1}+\frac{1}{\alpha(x)}+\sum_{n=1}^{\infty}\frac{1-2\alpha(x)}{(\alpha(x)+n)(1-\alpha(x)+n)}=\pi\rm
ctg\left(\frac{\pi\alpha(\it{x})}{2}\right).\end{align} Now, by
combining (\ref{eq:3.9}) - (\ref{eq:3.21}) we get
\begin{align}\label{eq:3.22}&\limsup_{x\longrightarrow\infty}\frac{\alpha(\it{x})}{c(\it{x})}(1+x)^{\alpha(x)}\left(\rm
\it{A}(\it{x})+\it{C}(\it{x})\right)\nonumber\\
&=\limsup_{x\longrightarrow\infty}\left(\frac{\alpha(\it{x})}{c(\it{x})}(1+x)^{\alpha(x)-1}\beta(x)+\pi\rm{ctg}\left(\frac{\pi\alpha(\it{x})}{\rm{2}}\right)\right).\end{align}
Finally, by combining (\ref{eq:3.8}), (\ref{eq:3.9}),
(\ref{eq:3.22}) and  assumption  (\ref{eq:1.2}) we get
\begin{align*}\limsup_{x\longrightarrow\infty}\frac{\alpha(x)}{c(x)}(1+x)^{\alpha(x)}\tilde{\mathcal{A}}V(x)<0.\end{align*}
The case when $x<0$ is treated in the same way. Therefore,
 we have proved the desired result.
\end{proof}

\begin{proof}[\textbf{Proof of  Theorem} \ref{tm1.1} (ii)]
The proof is divided  in three steps.

\textbf{Step 1.} In the first step we explain our strategy of the
proof. Let $\varphi\in C^{2}(\R)$ be an arbitrary nonnegative
function such that $\varphi(x)=|x|$, for $|x|>1$, and
$\varphi(x)\leq|x|$, for $|x|\leq1.$  Let $\theta\in(0,1)$ be
arbitrary and  let us define the function $V:\R\longrightarrow\R_+$
by the formula
$$V(x):=1-(1+\varphi(x))^{-\theta}.$$
Clearly, $V\in C^{2}(\R)$ and the level set  $C_V(r)=\{x:V(x)\leq
r\}$ is a compact set for all levels $0\leq r<1.$ Furthermore, since
the function $V(x)$ is bounded
$$\left|\int_{\{|y|>1\}}\left(V(x+y)-V(x)\right)\frac{c(x)}{|y|^{\alpha(x)+1}}dy\right|<\infty$$
holds for all $x\in\R.$ Hence, by the relation (\ref{eq:2.1}),
$V\in\mathcal{D}(\tilde{\mathcal{A}})$ and for the function
$\tilde{\mathcal{A}}V(x)$ we can take the function
$\mathcal{A}^{\alpha}V(x)$, where
                                                    $\mathcal{A}^{\alpha}$
                                                    is again
                                                    the
                                                    infinitesimal
                                                    generator of the
                                                    stable-like
                                                    process
                                                    $\process{X^{\alpha}}$
                                                    given by (\ref{eq:1.1}).
 In the sequel we show that  there exists $0<r_0<1$, such that $\tilde{\mathcal{A}}V(x)\geq0$ for all $x\in
\left(C_V(r_0)\right)^{c}.$ Clearly, $\sup_{x\in
C_V(r_0)}|\tilde{\mathcal{A}}V(x)|<\infty.$ Thus, the desired result
follows from Theorem \ref{tm2.3} (ii). Note that  for the sets
$C,D\subseteq\R$, defined in Theorem \ref{tm2.3} (ii), we can take
$C:=C_V(r_0)$ and $D$ is an arbitrary closed set satisfying
$D\subseteq C^{c}$ and $\lambda(D)>0$.  Now,  from  the continuity
of the function $V(x)$, we have
$$\sup_{x\in C}V(x)<\inf_{x\in D} V(x).$$
In order to prove the existence of such $r_0$, since
 $C_V(r)\uparrow\R$,
when $r\nearrow1$, it suffices to show that
\begin{align}\label{eq:3.23}\liminf_{|x|\longrightarrow\infty}\frac{\alpha(x)}{\theta
c(x)}(1+|x|)^{\alpha(x)+\theta}\tilde{\mathcal{A}}V(x)>0.\end{align}

\textbf{Step 2.} In the second step we find  more appropriate
expression for the function $\tilde{\mathcal{A}}V(x).$  We have
\begin{align*}\tilde{\mathcal{A}}V(x)=\mathcal{A}^{\alpha}V(x)
=&\beta(x)V'(x)+\int_{\{|y|\leq1\}}\left(V(x+y)-V(x)-V'(x)y\right)\frac{c(x)}{|y|^{\alpha(x)+1}}dy\\&+\int_{\{|y|>1\}}\left(V(x+y)-V(x)\right)\frac{c(x)}{|y|^{\alpha(x)+1}}dy.
\end{align*}
Let us define \begin{align*} A(x):&=\beta(x)V'(x)\\
B(x):&=\int_{\{|y|\leq1\}}\left(V(x+y)-V(x)-V'(x)y\right)\frac{c(x)}{|y|^{\alpha(x)+1}}dy\\
C(x):&=\int_{\{|y|>1\}}\left(V(x+y)-V(x)\right)\frac{c(x)}{|y|^{\alpha(x)+1}}dy.\end{align*}
Hence, in order to prove (\ref{eq:3.23}) it suffices to prove
\begin{align}\label{eq:3.24}&\liminf_{|x|\longrightarrow\infty}\frac{\alpha(x)}{\theta c(x)}(1+|x|)^{\alpha(x)+\theta}(A(x)+B(x)+C(x))>\rm{0}.\end{align}
Furthermore, for  $x>0$ large enough  we have
\begin{align*}
A(x)&=\theta\beta(\it{x})(\rm{1}+\it{x})^{-\theta-\rm{1}}\\
B(x)&=\int_{-1}^{1}\left((1+x)^{-\theta}-(1+x+y)^{-\theta}-\theta
(1+x)^{-\theta-1}y\right)\frac{c(x)}{|y|^{\alpha(x)+1}}dy
\end{align*}
and
\begin{align*}C(x)=&\int_{\{|y|>1\}}\left((1+x)^{-\theta}-(1+\varphi(x+y))^{-\theta}\right)\frac{c(x)}{|y|^{\alpha(x)+1}}dy\\
=&(1+x)^{-\theta}\int_{-\infty}^{-x-1}\left(1-\left(\frac{1-x-y}{1+x}\right)^{-\theta}\right)\frac{c(x)}{|y|^{\alpha(x)+1}}dy\\
&+(1+x)^{-\theta}\int_{-x-1}^{-x+1}\left(1-\left(\frac{1+\varphi(x+y)}{1+x}\right)^{-\theta}\right)\frac{c(x)}{|y|^{\alpha(x)+1}}dy\\
&+(1+x)^{-\theta}\int_{-x+1}^{-1}\left(1-\left(1+\frac{y}{1+x}\right)^{-\theta}\right)\frac{c(x)}{|y|^{\alpha(x)+1}}dy\\
&+(1+x)^{-\theta}\int_{1}^{\infty}\left(1-\left(1+\frac{y}{1+x}\right)^{-\theta}\right)\frac{c(x)}{|y|^{\alpha(x)+1}}dy.\end{align*}
By restricting the function $1-(1+t)^{-\theta}$ to intervals
$(-1,1)$ and $[1,\infty)$, and using its Taylor expansion, that is,
$$1-(1+t)^{-\theta}=-\displaystyle\sum_{i=1}^{\infty}{-\theta \choose i} t^{i},$$
for $t\in(-1,1)$, where $${-\theta \choose
i}=\frac{-\theta(-\theta-1)\cdots(-\theta-i+1)}{i!},$$ we get
\begin{align*}C(x)=&(1+x)^{-\theta}\int_{1+x}^{\infty}\left(1-\left(\frac{1-x+y}{1+x}\right)^{-\theta}\right)\frac{c(x)}{y^{\alpha(x)+1}}dy\\
&+(1+x)^{-\theta}\int_{x-1}^{1+x}\left(1-\left(\frac{1+\varphi(x-y)}{1+x}\right)^{-\theta}\right)\frac{c(x)}{y^{\alpha(x)+1}}dy\\
&-(1+x)^{-\theta}\displaystyle\sum_{i=1}^{\infty}{-\theta \choose
i}\frac{(-1)^{i}c(x)}{(1+x)^{i}}\displaystyle\int_{1}^{x-1}y^{i-\alpha(x)-1}dy\\
&-(1+x)^{-\theta}\displaystyle\sum_{i=1}^{\infty}{-\theta \choose
i}\frac{c(x)}{(1+x)^{i}}\displaystyle\int_{1}^{1+x}y^{i-\alpha(x)-1}dy\\
&+(1+x)^{-\theta}\displaystyle\int_{1+x}^{\infty}\left(1-\left(1+\frac{y}{1+x}\right)^{-\theta}\right)\frac{c(x)}{y^{\alpha(x)+1}}dy.\end{align*}
Let us put
\begin{align*}
C_1(x):=&\frac{\alpha(x)}{\theta}(1+x)^{\alpha(x)}\Bigg[\int_{1+x}^{\infty}\left(1-\left(\frac{1-x+y}{1+x}\right)^{-\theta}\right)\frac{dy}{y^{\alpha(x)+1}}\\
&\hspace{3cm}+\int_{1+x}^{\infty}\left(1-\left(1+\frac{y}{1+x}\right)^{-\theta}\right)\frac{dy}{y^{\alpha(x)+1}}\Bigg]\\
C_2(x):=&\frac{\alpha(x)}{\theta}(1+x)^{\alpha(x)}\int_{x-1}^{1+x}\left(1-\left(\frac{1+\varphi(x-y)}{1+x}\right)^{-\theta}\right)\frac{dy}{y^{\alpha(x)+1}}\\
C_3(x):=&\alpha(x)(1+x)^{\alpha(x)-1}\int_{x-1}^{1+x}y^{-\alpha(x)}dy\\
C_4(x):=&-\frac{\alpha(x)}{\theta}(1+x)^{\alpha(x)}\Bigg[\displaystyle\sum_{i=2}^{\infty}{-\theta
\choose
i}\frac{(-1)^{i}}{(1+x)^{i}}\displaystyle\int_{1}^{x-1}y^{i-\alpha(x)-1}dy\\&+\displaystyle\sum_{i=2}^{\infty}{-\theta
\choose
i}\frac{1}{(1+x)^{i}}\displaystyle\int_{1}^{1+x}y^{i-\alpha(x)-1}dy\Bigg].\end{align*}
 Hence, we find
\be\label{eq:3.25}\frac{\alpha(x)}{\theta
c(x)}(1+x)^{\alpha(x)+\theta}C(x)= C_1(x)+ C_2(x)+C_3(x)+C_4(x).\ee
 Further, by (\ref{eq:3.6}), we have
\begin{align*}C_1(x)=&\frac{2}{\theta}-\frac{\alpha(x)\,
_2F_1\left(\theta,\alpha(x)+\theta,1+\alpha(x)+\theta,-1\right)+\alpha(x)\,
_2F_1\left(\theta,\alpha(x)+\theta,1+\alpha(x)+\theta,1\right)}{\theta(\alpha(x)+\theta)}\\&-
\frac{\alpha(x)\,
_2F_1\left(\theta,\alpha(x)+\theta,1+\alpha(x)+\theta,\frac{x-1}{1+x}\right)-\alpha(x)\,
_2F_1\left(\theta,\alpha(x)+\theta,1+\alpha(x)+\theta,1\right)}{\theta(\alpha(x)+\theta)}.
\end{align*}
Next, by elementary computation, we have
\begin{align*}C_3(x)=\frac{\alpha(x)}{1-\alpha(x)}\left(1-\left(\frac{1+x}{x-1}\right)^{\alpha(x)-1}\right)\end{align*}
in the case when $\alpha(x)\neq1$,
\begin{align*}C_3(x)=\ln\left(\frac{1+x}{x-1}\right),\end{align*} in
the case when $\alpha(x)=1,$ and
\begin{align*}C_4(x)=-\frac{\alpha(x)}{\theta}\Bigg[&\sum_{i=1}^{\infty}{-\theta
\choose 2i}\frac{2}{2i-\alpha(x)}-\sum_{i=1}^{\infty}{-\theta
\choose
2i}\frac{2}{(2i-\alpha(x))(1+x)^{2i-\alpha(x)}}\\&+\sum_{i=2}^{\infty}{-\theta
\choose
i}\frac{(-1)^{i}}{i-\alpha(x)}\left(\left(\frac{x-1}{1+x}\right)^{i-\alpha(x)}-1\right)\Bigg].\end{align*}

\textbf{Step 3.} In the third step we prove
$$\liminf_{x\longrightarrow\infty}\frac{\alpha(x)}{\theta c(x)}(1+x)^{\alpha(x)+\theta}\tilde{\mathcal{A}}V(x)=\liminf_{x\longrightarrow\infty}\frac{\alpha(x)}{\theta c(x)}(1+x)^{\alpha(x)+\theta}(A(x)+B(x)+C(x))>0.$$
First, by the mean value theorem, we have
\begin{align}\label{eq:3.26}\lim_{x\longrightarrow\infty}\frac{\alpha(x)}{\theta
c(x)}(1+x)^{\alpha(x)+\theta}B(x)=0,\end{align} and, by Lemma
\ref{lm3.2}, we have
\begin{align}\label{eq:3.27}\lim_{x\longrightarrow\infty}C_2(x)=0\end{align}
and
\begin{align}\label{eq:3.28}\lim_{x\longrightarrow\infty}C_3(x)=0.\end{align}
Further, since
$0<\inf\{\alpha(x):x\in\R\}\leq\sup\{\alpha(x):x\in\R\}<2$, from
(\ref{eq:3.5}) and  the dominated convergence theorem, it follows
\begin{align}\label{eq:3.29}\lim_{x\longrightarrow\infty}\frac{\alpha(x)\,
_2F_1\left(\theta,\alpha(x)+\theta,1+\alpha(x)+\theta,\frac{x-1}{1+x}\right)-\alpha(x)\,
_2F_1\left(\theta,\alpha(x)+\theta,1+\alpha(x)+\theta,1\right)}{\theta(\alpha(x)+\theta)}=0\end{align}
and
\begin{align}\label{eq:3.30}\lim_{x\longrightarrow\infty}\Bigg[-\sum_{i=1}^{\infty}{-\theta
\choose
2i}\frac{2}{(2i-\alpha(x))(1+x)^{2i-\alpha(x)}}+\sum_{i=2}^{\infty}{-\theta
\choose
i}\frac{(-1)^{i}}{i-\alpha(x)}\left(\left(\frac{x-1}{1+x}\right)^{i-\alpha(x)}-1\right)\Bigg]=0.\end{align}
Thus, by combining  (\ref{eq:3.24}) - (\ref{eq:3.30}), we have
\begin{align*}&\liminf_{x\longrightarrow\infty}\frac{\alpha(x)}{\theta
c(x)}(1+x)^{\alpha(x)+\theta}\tilde{\mathcal{A}}V(x)\\&
=\liminf_{x\longrightarrow\infty}\Bigg[\frac{\alpha(x)}{c(x)}(1+x)^{\alpha(x)-1}\beta(x)-\frac{\alpha(x)}{\theta}\sum_{i=1}^{\infty}{-\theta
\choose 2i}\frac{2}{2i-\alpha(x)}+\frac{2}{\theta}\\
&\hspace{2cm}-\frac{\alpha(x)\,
_2F_1\left(\theta,\alpha(x)+\theta,1+\alpha(x)+\theta,-1\right)+\alpha(x)\,
_2F_1\left(\theta,\alpha(x)+\theta,1+\alpha(x)+\theta,1\right)}{\theta(\alpha(x)+\theta)}\Bigg].
\end{align*}
Next, it can be proved that the function
$$\theta\longmapsto-\frac{\alpha}{\theta}\sum_{i=1}^{\infty}{-\theta
\choose 2i}\frac{2}{2i-\alpha}+\frac{2}{\theta}-\frac{\alpha\,
_2F_1\left(\theta,\alpha+\theta,1+\alpha+\theta,-1\right)+\alpha\,
_2F_1\left(\theta,\alpha+\theta,1+\alpha+\theta,1\right)}{\theta(\alpha+\theta)}$$
is strictly decreasing, hence we choose $\theta$ close to zero.
From, (\ref{eq:3.2}),  (\ref{eq:3.3}) and (\ref{eq:3.4}), we get
\begin{align*}&\lim_{\theta\longrightarrow0}\left(-\frac{\alpha}{\theta}\sum_{i=1}^{\infty}{-\theta
\choose 2i}\frac{2}{2i-\alpha}+\frac{2}{\theta}-\frac{\alpha\,
_2F_1\left(\theta,\alpha+\theta,1+\alpha+\theta,-1\right)+\alpha\,
_2F_1\left(\theta,\alpha+\theta,1+\alpha+\theta,1\right)}{\theta(\alpha+\theta)}\right)\\
&=\pi\rm{ctg}\left(\frac{\pi\alpha}{2}\right).\end{align*} Now, the
claim follows from condition (\ref{eq:1.3}).
 The case when $x<0$ is treated in the same way.  Therefore,
 we have proved the desired result.
\end{proof}

In the case when $\limsup_{|x|\longrightarrow\infty}\alpha(x)<1$, we
also give an alternative, more probabilistic, proof of Theorem
\ref{tm1.1} (ii).
\begin{proof}[\textbf{Proof of  Theorem} \ref{tm1.1} (ii)] Let $\chain{X}$ be a Markov chain on the real line
given by the transition kernel $p(x,dy):=f_x(y-x)dy,$ where $f_x(y)$
is the density function of the stable distribution with
characteristic exponent
$p(x;\xi)=-i\beta(x)\xi+\gamma(x)|\xi|^{\alpha(x)}.$ Hence, the
chain $\chain{X}$ jumps from the state $x$ by the stable
distribution with the density function $f_x(y)$. By
\cite[Proposition 5.5 and Theorem 1.4]{sandric-rectrans}, the chain
$\chain{X}$ is transient. Further, let $\chain{X^{m}}$, $m\in\N$, be
a sequence of Markov chains on the real line given by  transition
kernels $p_m(x,dy):=f^{m}_x(y-x)dy,$ $m\in\N$, where $f^{m}_x(y)$ is
the density function of the stable distribution with characteristic
exponent $p_m(x;\xi):=\frac{1}{m}p(x;\xi).$ Then, by
\cite[Proposition 2.13]{sandric-periodic}, all  chains
$\chain{X^{m}}$, $m\in\N$, are transient as well. Further, by
\cite{bjoern-rene-approx}, we have $$X^{m}_{[
m\cdot]}\stackrel{\hbox{\scriptsize d}}{\longrightarrow}
X^{\alpha}_{\cdot},\ \ \textrm{as}\ \ m\longrightarrow\infty,$$
where $[ x]$ denotes the integer part of $x$ and
$\stackrel{\hbox{\scriptsize d}}{\longrightarrow}$ denotes the
convergence in the space of c\`adl\`ag functions equipped with the
Skorohod topology. Hence, the stable-like process
$\process{X^{\alpha}}$ can be approximated by c\`adl\`ag ``versions"
of chains $\chain{X^{m}}$, $m\in\N$. Therefore, all  we have to show
is that this approximation keeps the transience property. But this
fact follows from \cite[Proposition 2.4]{sandric-periodic}.
\end{proof}

\begin{proof}[\textbf{Proof of  Theorem} \ref{tm1.2}]
We use a similar strategy as in Theorem  \ref{tm1.1} (ii). The proof
is divided  in three steps.

\textbf{Step 1.} In the first step we explain  our strategy of the
proof. Let $\varphi\in C^{2}(\R)$ be an arbitrary nonnegative
function such that $\varphi(x)=|x|$, for $|x|>1$, and
$\varphi(x)\leq|x|$, for $|x|\leq1$, and let $\theta\in(1,\alpha)$
be arbitrary (recall that $1<\alpha=\inf\{\alpha(x):x\in\R\}$) and
let us define the function $V:\R\longrightarrow\R_+$ by the formula
$$V(x):=(\varphi(x))^{\theta}.$$
Clearly, $V\in C^{2}(\R)$ and the level set  $C_V(r)=\{x:V(x)\leq
r\}$ is a compact set for all levels $r\geq0.$ Furthermore, since
$\theta<\inf\{\alpha(x):x\in\R\}$, we have
$$\left|\int_{\{|y|>1\}}\left(V(x+y)-V(x)\right)\frac{c(x)}{|y|^{\alpha(x)+1}}dy\right|<\infty$$
for all $x\in\R.$ Hence, by the relation (\ref{eq:2.1}),
$V\in\mathcal{D}(\tilde{\mathcal{A}})$ and for the function
$\tilde{\mathcal{A}}V(x)$ we can take the function
$\mathcal{A}^{\alpha}V(x)$, where
                                                    $\mathcal{A}^{\alpha}$ is
                                                    the
                                                    infinitesimal
                                                    generator of the
                                                    stable-like
                                                    process
                                                    $\process{X^{\alpha}}$
                                                    given by (\ref{eq:1.1}).

In the sequel we show that  there exists $r_0>0$, large enough, such
that $\tilde{\mathcal{A}}V(x)\leq-1$ for all $x\in
\left(C_V(r_0)\right)^{c}.$ Clearly, $\sup_{x\in
C_V(r_0)}|\tilde{\mathcal{A}}V(x)|<\infty.$ Thus, the desired result
follows from Theorem \ref{tm2.3} (iii). In order to see this, since
$C_V(r)\uparrow\R$, when $r\nearrow\infty$, it suffices to show that
\begin{align}\label{eq:3.31}\displaystyle\limsup_{\theta\longrightarrow\alpha}\displaystyle\limsup_{|x|\longrightarrow\infty}\frac{\alpha(x)}{\theta c(x)}|x|^{\alpha(x)-\theta}(\tilde{\mathcal{A}}V(x)+1)<0.\end{align}
We have
\begin{align*}\tilde{\mathcal{A}}V(x)+1=\mathcal{A}V(x)+1
=&\beta(x)V'(x)+\int_{\{|y|\leq1\}}\left(V(x+y)-V(x)-V'(x)y\right)\frac{c(x)}{|y|^{\alpha(x)+1}}dy\\&+\int_{\{|y|>1\}}\left(V(x+y)-V(x)\right)\frac{c(x)}{|y|^{\alpha(x)+1}}dy+1.
\end{align*}
Let us define \begin{align*} A(x):&=\beta(x)V'(x)+1\\
B(x):&=\int_{\{|y|\leq1\}}\left(V(x+y)-V(x)-V'(x)y\right)\frac{c(x)}{|y|^{\alpha(x)+1}}dy\\
C(x):&=\int_{\{|y|>1\}}\left(V(x+y)-V(x)\right)\frac{c(x)}{|y|^{\alpha(x)+1}}dy.\end{align*}
Hence, in order to prove (\ref{eq:3.31}) it suffices to prove
\begin{align}\label{eq:3.32}\limsup_{\theta\longrightarrow\alpha}\limsup_{|x|\longrightarrow\infty}\frac{\alpha(x)}{\theta c(x)}|x|^{\alpha(x)-\theta}(A(x)+B(x)+C(x))<0.\end{align}
Furthermore, for  $x>0$ large enough we have
\begin{align*}
A(x)&=\theta\beta(x)x^{\theta-1}+1\\
B(x)&=\int_{-1}^{1}\left((x+y)^{\theta}-x^{\theta}-\theta
x^{\theta-1}y\right)\frac{c(x)}{|y|^{\alpha(x)+1}}dy
\end{align*}
and
\begin{align*}C(x)=&\int_{\{|y|>1\}}\left((\varphi(x+y))^{\theta}-x^{\theta}\right)\frac{c(x)}{|y|^{\alpha(x)+1}}dy\\
=&x^{\theta}\int_{-\infty}^{-x-1}\left(\left(-\frac{y}{x}-1\right)^{\theta}-1\right)\frac{c(x)}{|y|^{\alpha(x)+1}}dy\\
&+x^{\theta}\int_{-x-1}^{-x+1}\left(\left(\frac{\varphi(x+y)}{x}\right)^{\theta}-1\right)\frac{c(x)}{|y|^{\alpha(x)+1}}dy\\
&+x^{\theta}\int_{-x+1}^{-1}\left(\left(1+\frac{y}{x}\right)^{\theta}-1\right)\frac{c(x)}{|y|^{\alpha(x)+1}}dy\\
&+x^{\theta}\int_{1}^{\infty}\left(\left(1+\frac{y}{x}\right)^{\theta}-1\right)\frac{c(x)}{|y|^{\alpha(x)+1}}dy.\end{align*}
By restricting the function $(1+t)^{\theta}-1$ to intervals $(-1,1)$
and $[1,\infty)$, and using its Taylor expansion, that is,
$$(1+t)^{\theta}-1=\displaystyle\sum_{i=1}^{\infty}{\theta \choose i} t^{i},$$
for $t\in(-1,1),$ we get
\begin{align*}C(x)=&x^{\theta}\int_{1+x}^{\infty}\left(\left(\frac{y}{x}-1\right)^{\theta}-1\right)\frac{c(x)}{y^{\alpha(x)+1}}dy\\
&+x^{\theta}\int_{x-1}^{1+x}\left(\left(\frac{\varphi(x-y)}{x}\right)^{\theta}-1\right)\frac{c(x)}{y^{\alpha(x)+1}}dy\\
&+x^{\theta}\displaystyle\sum_{i=1}^{\infty}{\theta \choose
i}\frac{(-1)^{i}c(x)}{x^{i}}\displaystyle\int_{1}^{x-1}y^{i-\alpha(x)-1}dy\\
&+x^{\theta}\displaystyle\sum_{i=1}^{\infty}{\theta \choose
i}\frac{c(x)}{x^{i}}\displaystyle\int_{1}^{x}y^{i-\alpha(x)-1}dy\\
&+x^{\theta}\displaystyle\int_{x}^{\infty}\left(\left(1+\frac{y}{x}\right)^{\theta}-1\right)\frac{c(x)}{y^{\alpha(x)+1}}dy.\end{align*}
Let us put
\begin{align*}
C_1(x):=&\frac{\alpha(x)}{\theta}x^{\alpha(x)}\Bigg[\int_{1+x}^{\infty}\left(\left(\frac{y}{x}-1\right)^{\theta}-1\right)\frac{dy}{y^{\alpha(x)+1}}+\int_{x}^{\infty}\left(\left(1+\frac{y}{x}\right)^{\theta}-1\right)\frac{dy}{y^{\alpha(x)+1}}\Bigg]\\
C_2(x):=&\frac{\alpha(x)}{\theta}x^{\alpha(x)}\int_{x-1}^{1+x}\left(\left(\frac{\varphi(x+y)}{x}\right)^{\theta}-1\right)\frac{dy}{y^{\alpha(x)+1}}\\
C_3(x):=&\alpha(x)x^{\alpha(x)-1}\int_{x-1}^{x}y^{-\alpha(x)}dy\\
C_4(x):=&\frac{\alpha(x)}{\theta}x^{\alpha(x)}\Bigg[\displaystyle\sum_{i=2}^{\infty}{\theta
\choose
i}\frac{(-1)^{i}}{x^{i}}\displaystyle\int_{1}^{x-1}y^{i-\alpha(x)-1}dy+\displaystyle\sum_{i=2}^{\infty}{\theta
\choose
i}\frac{1}{x^{i}}\displaystyle\int_{1}^{x}y^{i-\alpha(x)-1}dy\Bigg].\end{align*}
 Hence, we find
\be\label{eq:3.33}\frac{\alpha(x)}{\theta
c(x)}x^{\alpha(x)-\theta}C(x)= C_1(x)+ C_2(x)+C_3(x)+C_4(x).\ee
 Further, by (\ref{eq:3.6}), we have
\begin{align*}C_1(x)=&-\frac{2}{\theta}+\frac{\alpha(x)\,
_2F_1\left(-\theta,\alpha(x)-\theta,1+\alpha(x)-\theta,1\right)+\alpha(x)\,
_2F_1\left(-\theta,\alpha(x)-\theta,1+\alpha(x)-\theta,-1\right)}{\theta(\alpha(x)-\theta)}\\&+\frac{1}{\theta}\left(1-\left(\frac{x}{1+x}\right)^{\alpha(x)}\right)
-\frac{\alpha(x)\,
_2F_1\left(-\theta,\alpha(x)-\theta,1+\alpha(x)-\theta,1\right)}{\theta(\alpha(x)-\theta)}\\
&+\alpha(x)\left(\frac{x}{1+x}\right)^{\alpha(x)-\beta}
\frac{_2F_1\left(-\theta,\alpha(x)-\theta,1+\alpha(x)-\theta,\frac{x}{1+x}\right)}{\theta(\alpha(x)-\theta)}.
\end{align*}
Next, by elementary computation, we have
\begin{align*}C_3(x)=\frac{\alpha(x)}{1-\alpha(x)}\left(1-\left(\frac{x}{x-1}\right)^{\alpha(x)-1}\right)\end{align*}
and
\begin{align*}C_4(x)=\frac{\alpha(x)}{\theta}\Bigg[&\sum_{i=1}^{\infty}{\theta
\choose 2i}\frac{2}{2i-\alpha(x)}-\sum_{i=1}^{\infty}{\theta \choose
2i}\frac{2}{(2i-\alpha(x))x^{2i-\alpha(x)}}\\&+\sum_{i=2}^{\infty}{\theta
\choose
i}\frac{(-1)^{i}}{i-\alpha(x)}\left(\left(\frac{x-1}{x}\right)^{i-\alpha(x)}-1\right)\Bigg].\end{align*}

\textbf{Step 3.} In the third step we prove
$$\limsup_{\theta\longrightarrow\alpha}\limsup_{x\longrightarrow\infty}\frac{\alpha(x)}{\theta c(x)}x^{\alpha(x)-\theta}\tilde{\mathcal{A}}V(x)=\limsup_{\theta\longrightarrow\alpha}\limsup_{x\longrightarrow\infty}\frac{\alpha(x)}{\theta c(x)}x^{\alpha(x)-\theta}(A(x)+B(x)+C(x))<0.$$
First, by the mean value theorem, we have
\begin{align}\label{eq:3.34}\lim_{x\longrightarrow\infty}\frac{\alpha(x)}{\theta
c(x)}x^{\alpha(x)-\theta}B(x)=0,\end{align} and, by Lemma
\ref{lm3.2}, we have
\begin{align}\label{eq:3.35}\lim_{x\longrightarrow\infty}C_2(x)=0\end{align}
and
\begin{align}\label{eq:3.35}\lim_{x\longrightarrow\infty}C_3(x)=0.\end{align}
Further, since
$0<\inf\{\alpha(x):x\in\R\}\leq\sup\{\alpha(x):x\in\R\}<2$, from
(\ref{eq:3.5}) and  the dominated convergence theorem, it follows
\begin{align}\label{eq:3.37}&\lim_{x\longrightarrow\infty}\Bigg[\frac{1}{\theta}\left(1-\left(\frac{x}{1+x}\right)^{\alpha(x)}\right)
-\frac{\alpha(x)\,
_2F_1\left(-\theta,\alpha(x)-\theta,1+\alpha(x)-\theta,1\right)}{\theta(\alpha(x)-\theta)}\nonumber\\
&\hspace{1.3cm}+\alpha(x)\left(\frac{x}{1+x}\right)^{\alpha(x)-\beta}
\frac{_2F_1\left(-\theta,\alpha(x)-\theta,1+\alpha(x)-\theta,\frac{x}{1+x}\right)}{\theta(\alpha(x)-\theta)}\Bigg]=0\end{align}
and
\begin{align}\label{eq:3.38}\lim_{x\longrightarrow\infty}\Bigg[-\sum_{i=1}^{\infty}{\theta
\choose
2i}\frac{2}{(2i-\alpha(x))x^{2i-\alpha(x)}}+\sum_{i=2}^{\infty}{\theta
\choose
i}\frac{(-1)^{i}}{i-\alpha(x)}\left(\left(\frac{x-1}{x}\right)^{i-\alpha(x)}-1\right)\Bigg]=0.\end{align}
Thus, by combining (\ref{eq:3.32}) - (\ref{eq:3.38}), we have
\begin{align*}&\limsup_{x\longrightarrow\infty}\frac{\alpha(x)}{\theta
c(x)}x^{\alpha(x)-\theta}\tilde{\mathcal{A}}V(x)\\&
=\limsup_{x\longrightarrow\infty}\Bigg[\frac{\alpha(x)}{c(x)}x^{\alpha(x)-1}\beta(x)+\frac{\alpha(x)}{\theta
c(x)}x^{\alpha(x)-\theta}+\frac{\alpha(x)}{\theta}\sum_{i=1}^{\infty}{\theta
\choose 2i}\frac{2}{2i-\alpha(x)}-\frac{2}{\theta}\\
&\hspace{2cm}+\frac{\alpha(x)\,
_2F_1\left(-\theta,\alpha(x)-\theta,1+\alpha(x)-\theta,-1\right)+\alpha(x)\,
_2F_1\left(-\theta,\alpha(x)-\theta,1+\alpha(x)-\theta,1\right)}{\theta(\alpha(x)-\theta)}\Bigg]\\
&=\limsup_{x\longrightarrow\infty}\left(\frac{\alpha(x)}{c(x)}x^{\alpha(x)-1}\beta(x)+\frac{\alpha(x)}{
\theta c(x)}x^{\alpha(x)-\theta}+E(\alpha(x),\theta)\right).
\end{align*} Clearly, because of the term $x^{\alpha(x)-\theta}$, we
choose $\theta$ close to $\alpha$.
 Now,  from (\ref{eq:1.3}), it follows $$\limsup_{\theta\longrightarrow\alpha}\limsup_{x\longrightarrow\infty}\frac{\alpha(x)}{
 \theta c(x)}x^{\alpha(x)-\theta}\tilde{\mathcal{A}}V(x)<0.$$
 The case when $x<0$ is treated in the same way.  Therefore,
 we have proved the desired result.
\end{proof}

\begin{proof}[\textbf{Proof of  Corollary} \ref{c1.3}]
In the case when  $\alpha\neq2$, the claim easily follows from
Theorem \ref{tm1.1}. Further, in the case when $\alpha=2$, that is,
in the Brownian motion case, the corresponding infinitesimal
generator is given by $\mathcal{A}^{2}f(x)=\gamma f''(x)$ (recall
that the symbol (characteristic exponent) is given by
$p(\xi)=\gamma|\xi|^{2}$) and clearly
$C^{2}(\R)\subseteq\mathcal{D}(\tilde{\mathcal{A}}).$ Thus, for any
$f\in C^{2}(\R)$, for the function $\tilde{\mathcal{A}}f(x)$ we can
take the function $\mathcal{A}^{2}f(x)$. Now, by taking again
$V(x)=\log(1+\varphi(x))$ for the test function, where $\varphi\in
C^{2}(\R)$ is an arbitrary nonnegative function such that
$\varphi(x)=|x|$ for all $|x|>1$, we get
$$\tilde{\mathcal{A}}V(x)=\mathcal{A}^{2}V(x)=\frac{-\gamma}{(1+|x|)^{2}}$$ for all
$|x|>1$, that is,  the Brownian motion is recurrent.
\end{proof}

\begin{proof}[\textbf{Proof of  Theorem} \ref{tm2.4}]
Let $\theta\in(1,\alpha)$  be arbitrary (recall that
$1<\alpha=\inf_{x\in\R}\alpha(x)$). Further, let  there exist a
Borel measurable function $f(x)\geq1$ and strictly increasing,
nonnegative and convex function $\phi(x)$, such that
$|x|^{\theta}=\phi(f(x))$ for all $|x|$ large enough,
 then, by \cite[Theorem 5.3 (ii)]{meyn-tweedie-III}, the stable-like
 process $\process{X^{\alpha}}$ is  $f$-ergodic if
$$\tilde{\mathcal{A}}\,|x|^{\theta}\leq-f(x)$$ holds for all $|x|$ large enough.
Now, by  repeating the proof of Theorem \ref{tm1.2} and applying
condition (\ref{eq:2.2}), we have
\begin{align*}&\limsup_{|x|\longrightarrow\infty}\frac{\alpha(x)}{\theta
c(x)}|x|^{\alpha(x)-\theta}(\tilde{\mathcal{A}}V(x)+f(x))\\&=
\limsup_{|x|\longrightarrow\infty}\left(\rm
sgn(\it{x})\frac{\alpha(\it{x})}{c(\it{x})}|x|^{\alpha(x)-\rm{1}}\beta(\it{x})+\frac{\alpha(\it{x})}{\theta
c(\it{x})}|x|^{\alpha(x)-\theta}f(x)+E(\alpha(x),\theta)\right)<0.\end{align*}
Therefore, we have proved the desired result.
\end{proof}

\section*{Acknowledgement} The author would like to thank
the anonymous reviewer  for  careful reading of the paper and for
helpful comments that led to improvement of the presentation.

\bibliographystyle{alpha}
\bibliography{References}

\end{document}